\pdfoutput=1
\documentclass{amsart}
\usepackage{research}

\DeclareMathOperator{\im}{im}
\DeclareMathOperator{\ind}{ind}
\DeclareMathOperator{\Int}{int}

\newcommand{\D}{\mathbf{D}}

\usetikzlibrary{decorations.markings}
\graphicspath{{./graphics/}}

\begin{document}
\title{Splitting symplectic fillings}
\author{Austin Christian and Michael Menke}
\begin{abstract}
We generalize the mixed tori which appear in the second author's JSJ-type decomposition theorem for symplectic fillings of contact manifolds.  Mixed tori are convex surfaces in contact manifolds which may be used to decompose symplectic fillings.  We call our more general surfaces splitting surfaces, and show that the decomposition of symplectic fillings continues to hold.  Specifically, given a strong or exact symplectic filling of a contact manifold which admits a splitting surface, we produce a new symplectic manifold which strongly or exactly fills its boundary, and which is related to the original filling by Liouville surgery.
\end{abstract}
\maketitle

\section{Introduction}
Contact geometry is a close relative of symplectic geometry, and one manifestation of this relationship is the tendency for symplectic manifolds-with-boundary to endow their boundaries with contact structures.  For instance, suppose $(W,\omega)$ is a compact symplectic manifold which admits a Liouville vector field near its boundary.  That is, there is a vector field $Z$ on $W$ pointing out of $\partial W$ with the property that $\mathcal{L}_Z\omega=\omega$ in some neighborhood of $M=\partial W$.  Then $M$ inherits an orientation from $W$ and $\lambda:=\iota_Z\omega$ determines a co-oriented contact structure $\xi:=\ker(\lambda|_M)$ on $M$.  In this case say that $(W,\omega)$ is a \emph{strong symplectic filling} of the contact manifold $(M,\xi)$.\\

It is natural to wonder about the extent to which this construction is reversible.  That is, we begin with a fixed contact manifold $(M,\xi)$ and ask existence and uniqueness questions about the strong symplectic fillings of this manifold.  Eliashberg and Gromov showed in \cite{eliashberg1991convex} that a fillable contact manifold must be tight, so the overtwisted contact manifolds immediately give a large class of manifolds which are not symplectically fillable.  In \cite{etnyre2002tight} Etnyre and Honda showed that while tightness is necessary for fillability, it is not sufficient.  Another early result, due to Eliashberg (\cite{eliashberg1990filling}) and Gromov (\cite{gromov1985pseudo}), says that symplectic fillings of the standard 3-sphere $(S^3,\xi_{std})$ are unique up to symplectic deformation equivalence and blowup.  If we further require the filling to be \emph{exact}, meaning that $\mathcal{L}_Z\omega=\omega$ on all of $W$, then $(S^3,\xi_{std})$ in fact has a unique filling up to symplectomorphism.\\

A number of contact 3-manifolds have seen their exact fillings classified up to symplectomorphism, symplectic deformation equivalence, or diffeomorphism.  Wendl showed in \cite{wendl2010strongly} that $(\mathbb{T}^3,\xi_1)$ has a unique exact filling up to symplectomorphism, where $\xi_1$ is the canonical contact structure on $ST^*\mathbb{T}^2$, and work of McDuff (\cite{mcduff1990structure}) and Lisca (\cite{lisca2008symplectic}) classified the exact fillings of lens spaces $(L(p,q),\xi_{std})$ up to diffeomorphism.  Some classification results also exist for higher-dimensional contact manifolds, but giving precise symplecto-geometric descriptions of higher-dimensional fillings is difficult.  The most famous result in high dimensions is probably the Eliashberg-Floer-McDuff theorem (\cite{mcduff1991symplectic}), which says that, up to diffeomorphism, $(S^{2n-1},\xi_{std})$ has a unique symplectically aspherical strong symplectic filling, for all $n\geq 3$.\\

In \cite{menke2018jsj}, the second author introduced the notion of a \emph{mixed torus} --- a special kind of convex torus --- in a contact 3-manifold, and showed that if $(M,\xi)$ admits a mixed torus, then we may construct from any strong symplectic filling $(W,\omega)$ of $(M,\xi)$ another symplectic manifold $(W',\omega')$ which strongly fills its boundary $(M',\xi')$.  Moreover, the relationship between $(W,\omega)$ and $(W',\omega')$ may be stated rather explicitly, with $(W,\omega)$ obtained from $(W',\omega')$ by Liouville surgery in a prescribed manner.  This allows us to leverage an understanding of the fillings of $(M',\xi')$ into information about the fillings of $(M,\xi)$.\\

In this note we consider higher-genus analogues of mixed tori, which we call \emph{splitting surfaces}.  We will give a precise definition of splitting surfaces in Section \ref{sec:background}, but a splitting surface of genus 1 is simply a mixed torus.  The purpose of this note is to show that the main theorem of \cite{menke2018jsj} continues to hold in any genus.

\begin{theorem}\label{main-theorem}
Let $(M,\xi)$ be a closed, co-oriented 3-dimensional contact manifold and let $(W,\omega)$ be a strong (respectively, exact) filling of $(M,\xi)$.  If $(M,\xi)$ admits a splitting surface $\Sigma$ of genus $g$, then there exists a symplectic manifold $(W',\omega')$ such that
\begin{enumerate}
	\item $(W',\omega')$ is a strong (respectively, exact) filling of its boundary $(M',\xi')$;
	
	\item there are Legendrian graphs $\Lambda_1,\Lambda_2\subset\partial W'$ with standard neighborhoods $N(\Lambda_1),N(\Lambda_2)$ such that
	\[
	M \simeq \left(\partial W' - \bigcup_{i=1}^{2}\Int(N(\Lambda_i))\right)/(\partial N(\Lambda_1)\sim \partial N(\Lambda_2)),
	\]
	where the boundaries $\partial N(\Lambda_i)$ are glued in such a way that their dividing sets and meridians are identified;\label{main-thm:decomp}
	
	\item $(W,\omega)$ can be recovered from $(W',\omega')$ by attaching a symplectic handle $(H_{R_+(\Sigma)},\omega_\beta)$ constructed from the positive region of $\Sigma$.
\end{enumerate}
\end{theorem}

The first use of mixed tori to classify symplectic fillings came in the form of \cite[Theorem 1.2]{menke2018jsj}, where it is shown that if $(M,\xi)$ is obtained from $(M_0,\xi_0)$ by Legendrian surgery along a Legendrian knot which has been stabilized both positively and negatively, then every exact filling of $(M,\xi)$ is obtained from an exact filling of $(M_0,\xi_0)$ by attaching a round symplectic 1-handle along the Legendrian knot.  In particular, this means that contact manifolds obtained from $(S^3,\xi_{std})$ by Legendrian surgery along twice-stabilized Legendrian knots have unique exact fillings.  The following is then obtained by repeatedly applying \cite[Theorem 1.2]{menke2018jsj}:

\begin{corollary}\label{corollary:links}
Let $\Lambda\subset(S^3,\xi_{std})$ be a linear chain of Legendrian unknots, so that Legendrian surgery along $\Lambda$ produces a tight lens space.  If each unknot has been stabilized both positively and negatively, then this lens space admits a unique exact filling up to symplectomorphism.
\end{corollary}

We mention this result here because the methods that were used to prove \cite[Theorem 1.2]{menke2018jsj} from \cite[Theorem 1.1]{menke2018jsj} could also be used to prove Corollary \ref{corollary:links} from Theorem \ref{main-theorem}.  An interesting (if vague) question is then the following: let $\Lambda\subset(S^3,\xi_{std})$ be a Legendrian link, and let $(M,\xi)$ be the result of Legendrian surgery along $\Lambda$.  Other than the condition listed in Corollary \ref{corollary:links}, are there topological properties of $\Lambda$ or configurations of stabilizations on its components which force $(M,\xi)$ to admit a splitting surface?  Under what circumstances does this yield a classification of the fillings of $(M,\xi)$?\\

Our strategy of proof for the main theorem follows in the tradition of Eliashberg's "filling by holomorphic disks," initiated in \cite{eliashberg1990filling}.  A splitting surface $\Sigma\subset(M,\xi)$ of genus $g$ gives us two surfaces in $M$ with genus 0 and $g+1$ boundary components, each of which can be lifted to a family of $J$-holomorphic curves in the symplectization of $M$.  If we have a filling $(W,\omega)$ of $(M,\xi)$, these families can be extended to a single 1-dimensional family of $J$-holomorphic curves in the completion $(\widehat{W},\widehat{\omega})$, and the geometric conditions on $\Sigma$ will control the topology of this family.  Removing a neighborhood of this family will lead us to the new symplectic manifold $(W',\omega')$.\\

In Section \ref{sec:background} we recall some useful definitions and results from contact geometry and give a definition of our splitting surfaces.  Section \ref{sec:main-theorem} contains the proof of Theorem \ref{main-theorem}.

\subsection*{Acknowledgements}  The authors would like to thank Ko Honda for a number of helpful conversations and suggestions during the completion of this project.

\section{Background}\label{sec:background}
Throughout this section we fix a closed contact 3-manifold $(M,\xi)$.

\subsection{Fillings of contact manifolds}\label{subsec:fillings}
As mentioned above, many symplectic manifolds endow their boundaries with contact structures, and there are various levels of compatibility between the symplectic and contact structures.  In the other direction, we say that our contact manifold $(M,\xi)$ is \emph{fillable} if it can be realized as the boundary of such a symplectic manifold.  We have the following definitions.

\begin{definition}
Fix a co-oriented contact manifold $(M,\xi)$ and suppose $(W,\omega)$ is a symplectic manifold with $\partial W=M$ as oriented manifolds.  We say that $(W,\omega)$ is
\begin{itemize}
	\item a \emph{weak symplectic filling} of $(M,\xi)$ if $\omega|_\xi>0$;
	\item a \emph{strong symplectic filling} of $(M,\xi)$ if there is a 1-form $\lambda$ on $W$ such that $\omega=d\lambda$ on some neighborhood of $\partial W$ and $\xi=\ker(\lambda|_{\partial W})$;
	\item an \emph{exact filling} of $(M,\xi)$ if there is a 1-form $\lambda$ on $W$ such that $\omega=d\lambda$ on all of $W$ and $\xi=\ker(\lambda|_{\partial W})$.
\end{itemize}
We say that $(M,\xi)$ is \emph{weakly symplectically fillable}, \emph{strongly symplectically fillable}, or \emph{exactly fillable} if it admits a weak symplectic, strong symplectic, or exact fillling, respectively.\\
\end{definition}

Certainly every exact filling is a strong filling and every strong filling is a weak filling, so we have inclusions
\[
\{\text{exactly fillable}\}
\subseteq
\{\text{strongly symplectically fillable}\}
\subseteq
\{\text{weakly symplectically fillable}\}.
\]
One hypothesis of Theorem \ref{main-theorem} is that our contact manifold $(M,\xi)$ admits a strong or exact filling $(W,\omega)$, so our manifolds will always be at least strongly fillable.

\subsection{Convex surfaces}\label{subsec:convex-surfaces}
We quickly recall the notion of convexity in contact topology, as explored by Giroux in \cite{giroux1991convexite}.  First, a \emph{contact vector field} on a contact 3-manifold $(M,\xi)$ is a vector field whose flow preserves $\xi$.  Notice that if $\lambda$ is a contact form for $\xi$ and $X$ is a contact vector field, then
\[
\mathcal{L}_X\lambda = g\lambda
\]
for some positive smooth function $g$, so flowing along $X$ produces conformal dilations of the contact form.  For this reason we say that a surface $\Sigma\subset (M,\xi)$ is \emph{convex} if there is a contact vector field for $(M,\xi)$ which is transverse to $\Sigma$.  An important observation is that convex surfaces exist in abundance.

\begin{theorem}[{\cite{giroux1991convexite}}]
Any closed surface in a contact manifold $(M,\xi)$ is $C^\infty$-close to a convex surface.
\end{theorem}

If $\Sigma\subset(M,\xi)$ is convex and $X$ is a contact vector field transverse to $\Sigma$, then the \emph{dividing set} of $\Sigma$ is
\[
\Gamma_\Sigma = \{p\in\Sigma~|~X(p)\in\xi_p\}.
\]
Three important observations about the multi-curve $\Gamma_\Sigma$ are
\begin{enumerate}
	\item $\Gamma_\Sigma$ divides $\Sigma$ into positive and negative regions: $\Sigma\setminus\Gamma_\Sigma=R_+(\Sigma)\sqcup R_-(\Sigma)$;
	\item $\Gamma_\Sigma$ is transverse to the characteristic foliation $\Sigma_\xi$ of $\Sigma$;
	\item $\Sigma$ admits a volume form $\omega$ and a vector field $Y$ so that $Y$ points transversely out of $R_+(\Sigma)$ along $\Gamma_\Sigma$, directs $\Sigma_\xi$, and dilates $\omega$ in the sense that $\pm\mathcal{L}_Y\omega>0$ on $R_{\pm}(\Sigma)$.
\end{enumerate}
These three characteristics determine $\Gamma_\Sigma$ up to isotopy, so we will refer to the dividing set $\Gamma_\Sigma$ and the regions $R_{\pm}(\Sigma)$ of a convex surface $\Sigma$ without reference to a particular contact vector field.

\subsection{Bypasses and stabilizations}\label{subsec:bypasses}
If $\Sigma\subset M$ is a convex surface, recall that a \emph{bypass} for $\Sigma$ is an oriented embedded half-disk $D$ such that
\begin{enumerate}
	\item $\partial D$ is the union of two Legendrian arcs $\alpha_1,\alpha_2$ which intersect at their endpoints;
	\item $D$ intersects $\Sigma$ transversely along $\alpha_1$;
	\item $D$ has positive elliptic tangencies at $\alpha_1\cap\alpha_2$, one negative elliptic tangency on the interior of $\alpha_1$, and only positive tangencies along $\alpha_2$, alternating between elliptic and hyperbolic;
	\item $\alpha_1$ intersects the dividing set $\Gamma_\Sigma$ exactly at the elliptic points of $\alpha_1$.
\end{enumerate}
We will refer to $\alpha_1\subset D$ as the \emph{attaching arc} for the bypass $D$, and we say that $D$ \emph{straddles} the component $c\subset\Gamma_\Sigma$ containing the negative elliptic tangency.\\

\begin{figure}
\centering
\begin{tikzpicture}[scale=1.35]
\draw (0,0) rectangle (3,3);
\draw [red,thick] (0,0.5) -- (3,0.5)
	  (0,1.5) -- (3,1.5)
	  (0,2.5) -- (3,2.5);
\draw (1.5,2.5) -- (1.5,0.5);
\node [right] at (1.5,2) {$\alpha$};

\draw (4,0) rectangle (7,3);
\draw [red,thick] (4,2.5) to [out=0,in=180] (7,0.5)
	  (4,1.5) .. controls (5.5,1.5) and (5.5,0.5) .. (4,0.5)
	  (7,2.5) .. controls (5.5,2.5) and (5.5,1.5) .. (7,1.5);
\end{tikzpicture}
\caption{On the left, the dividing set $\Gamma_{\Sigma_0}$ in a neighborhood the attaching arc $\alpha$.  On the right, the dividing set $\Gamma_{\Sigma_1}$.}
\label{fig:bypass-attachment}
\end{figure}
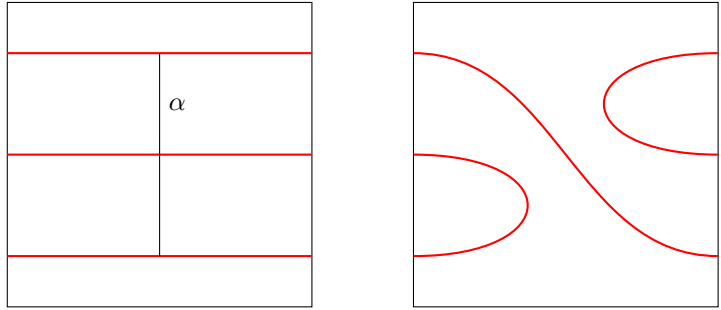

When a bypass $D$ for $\Sigma$ exists it is known that there is a neighborhood of $\Sigma\cup D$, diffeomorphic to $\Sigma\times[0,1]$, such that $\Sigma_i=\Sigma\times\{i\}$, $i=0,1$, are convex and the dividing set $\Gamma_{\Sigma_1}$ is obtained from $\Gamma_{\Sigma_0}$ by Honda's \emph{bypass attachment} operation, depicted in Figure \ref{fig:bypass-attachment}.  A bypass which does not change the dividing set is said to be \emph{trivial}.  The effect of bypass attachment on the dividing set of $\Sigma$ can also be seen through Giroux's contact handle decompositions.  The surface $\Sigma_1$ is obtained from $\Sigma$ by attaching a contact 1-handle and then a contact 2-handle in topologically canceling manner.  A detailed description of this process can be found in \cite[Section 3]{ozbagci2011contact}.\\

We are now prepared to define our splitting surfaces.

\begin{definition}
We call a closed, connected, oriented, convex surface $\Sigma\subset(M,\xi)$ of genus $g$ a \emph{splitting surface} if
\begin{enumerate}
	\item the regions $R_{\pm}(\Sigma)$ are planar, with $g+1$ boundary components $c_1,\ldots,c_{g+1}$;
	\item there exist bypasses $D^{\pm}_1,\ldots,D^\pm_g\subset(M,\xi)$, attached to $\Sigma$ along Legendrian arcs $\alpha^{\pm}_1,\ldots,\alpha^\pm_g$, with $\alpha^\pm_i$ straddling $c_i$ and having its endpoints on $c_{g+1}$;
	\item for $i=1,\ldots,g$, there is an arc $a_i\subset c_{g+1}$ which contains the endpoints of $\alpha_i^+$ and $\alpha_i^-$, and contains no endpoints of $\alpha_j^\pm$ for $j\neq i$;
	\item the bypasses $D^+_1,\ldots,D^+_g$ are attached from one side of $\Sigma$ and the bypasses $D^-_1,\ldots,D^-_g$ are attached from the other side.
\end{enumerate}
\end{definition}

\begin{figure}
\centering
\begin{tikzpicture}[scale=1.35]
\begin{scope}[decoration={markings,
	mark = at position 0.5 with {\arrow[very thick]{>}}
}]
\draw [postaction={decorate},thick] (-0.5,2) -- (2,2);

\draw [->] (2.35,2.2) -- (3.9,2.8);
\node [above] at (3.15,2.5) {$S_+$};
\draw [->] (2.35,1.8) -- (3.9,1.2);
\node [below] at (3.15,1.3) {$S_-$};

\draw [thick] (4.25,1) -- (4.75,1);
\draw [postaction={decorate},thick] (4.75,1) to [out=0,in=180] (6.75,0.5);
\draw [thick] (6.75,0.5) to [out=180,in=0] (5.75,1);
\draw [thick] (5.75,1) -- (7.25,1);

\draw [thick] (4.25,3) -- (5.75,3);
\draw [thick] (5.75,3) to [out=180,in=0] (4.75,2.5);
\draw [postaction={decorate},thick] (4.75,2.5) to [out=0,in=180] (6.75,3);
\draw [thick] (6.75,3) -- (7.25,3);
\end{scope}
\end{tikzpicture}
\caption{Stabilization of the $x$-axis in the front projection.}
\label{fig:stabilization}
\end{figure}
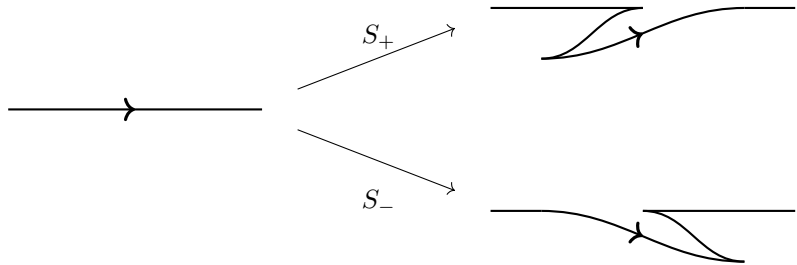

\subsection{Liouville hypersurfaces}\label{subsec:liouville}
The last statement of Theorem \ref{main-theorem} says that we can obtain our original symplectic filling $(W,\omega)$ from our new filling $(W',\omega')$ by attaching a symplectic handle.  The construction of the handle in question begins with the positive region $R_+(\Sigma)$ of our splitting surface, which is a \emph{Liouville hypersurface}.  In this section we want to review Avdek's definition (\cite{avdek2012liouville}) of Liouville hypersurfaces and produce the corresponding symplectic handle.\\

\begin{definition}
A \emph{Liouville domain} is a pair $(\Sigma_L,\beta)$, where
\begin{enumerate}
	\item $\Sigma_L$ is a smooth, compact manifold with boundary;
	\item $d\beta$ is a symplectic form on $\Sigma_L$;
	\item the vector field $X_\beta$ defined by $\iota_{X_\beta}d\beta=\beta$ points out of $\partial\Sigma_L$ transversely.
\end{enumerate}
We call $X_\beta$ the \emph{Liouville vector field} for $(\Sigma_L,\beta)$.\\
\end{definition}

\begin{definition}
Let $(M,\xi)$ be a contact 3-manifold and let $(\Sigma_L,\beta)$ be a 2-dimensional Liouville domain.  A \emph{Liouville embedding} $i\colon(\Sigma_L,\beta)\hookrightarrow(M,\xi)$ is an embedding for which there exists a contact form $\lambda$ on $(M,\xi)$ satisfying $i^*\lambda=\beta$.  We call the image of a Liouville embedding a \emph{Liouville hypersurface} and denote it by $(\Sigma_L,\beta)\subset(M,\xi)$.\\
\end{definition}

The standard example of a Liouville hypersurface is the positive region of a convex surface.  The following result says that these regions are in fact the source of all Liouville hypersurfaces.

\begin{proposition}[{\cite[Proposition 6.3]{avdek2012liouville}}]
A hypersurface $\Sigma_L\subset(M,\xi)$ is Liouville if and only if there is a convex hypersurface $\Sigma\subset(M,\xi)$ for which $\Sigma_L$ is $R_+(\Sigma)$ minus some collar neighborhood of $\partial R_+(\Sigma)$.
\end{proposition}

Given a Liouville hypersurface $(\Sigma_L,\beta)$, Avdek constructs a symplectic handle $(H_{\Sigma_L},\omega_\beta)$, and we summarize this construction here.  For full details see \cite{avdek2012liouville}.\\

The construction begins with a standard neighborhood $\mathcal{N}(\Sigma_L)$ of $(\Sigma_L,\beta)$ in $(M,\xi)$.  If $\lambda$ is a contact form for $(M,\xi)$ satisfying $\lambda|_{T\Sigma_L}=\beta$, then there is a neighborhood $N(\Sigma_L) = [-\epsilon,\epsilon]\times\Sigma_L$  with $\lambda|_{N(\Sigma_L)}=dz+\beta$, for some sufficiently small $\epsilon$.  This neighborhood will have corners at $\{\pm\epsilon\}\times\partial\Sigma_L$, but an edge-rounding process produces $\mathcal{N}(\Sigma_L)$, a neighborhood of $(\Sigma_L,\beta)$ with smooth, convex boundary.\\

With an abstract copy of this standard neighborhood in hand, consider the symplectic manifold
\[
(H_{\Sigma_L},\omega_\beta) = ([-1,1]\times\mathcal{N}(\Sigma),d\theta\wedge dz+d\beta),
\]
where $\theta$ and $z$ are the coordinates on $[-1,1]$ and $[-\epsilon,\epsilon]$, respectively.  This is the symplectic handle constructed from $(\Sigma_L,\beta)$.  There is a vector field $V_\beta=z\partial_z+X_\beta$ which points transversely out of $\partial H_{\Sigma_L}$ along $[-1,1]\times\partial\mathcal{N}(\Sigma_L)$ and whose flow dilates $\omega_\beta$.  This vector field can be perturbed so that it also points into $\partial H_{\Sigma_L}$ along $\{\pm 1\}\times\mathcal{N}(\Sigma_L)$, making this portion of $\partial H(\Sigma_L)$ concave while $[-1,1]\times\partial\mathcal{N}(\Sigma_L)$ is convex.\\

Let us also describe how Avdek attaches the symplectic handle $(H_{\Sigma_L},\omega_\beta)$ to a strong symplectic filling $(W,\omega)$.  For this attachment to be possible there must exist a pair of disjoint Liouville embeddings
\[
i_1\colon(\Sigma_L,\beta)\hookrightarrow (M,\xi)
\quad\text{and}\quad
i_2\colon(\Sigma_L,\beta)\hookrightarrow (M,\xi),
\]
where $(M,\xi)$ is the boundary of $(W,\omega)$.  These embeddings admit standard neighborhoods $\mathcal{N}(i_1(\Sigma_L))$ and $\mathcal{N}(i_2(\Sigma_L))$, each contactomorphic to $\mathcal{N}(\Sigma_L)$.  We form a sort of symplectic-filling-with-corners $W^\square$ by removing $\mathcal{N}(i_1(\Sigma_L))$ and $\mathcal{N}(i_2(\Sigma_L))$ from $(W,\omega)$ and attaching $H_{\Sigma_L}$ along $\{\pm 1\}\times\mathcal{N}(\Sigma_L)$.  Because $(W,\omega)$ is a strong filling of $(M,\xi)$, there is a Liouville vector field on $W$ pointing out of $\partial W$.  We glue $(H_{\Sigma_L},\omega_\beta)$ to $(W\setminus(\mathcal{N}(i_1(\Sigma_L))\cup\mathcal{N}(i_2(\Sigma_L))),\omega)$ in such a way that this vector field agrees with $V_\beta$ along $\{\pm 1\}\times \mathcal{N}(\Sigma_L)$.  The edges of $W^\square$ are then rounded to produce a new symplectic filling $(W',\omega')$.  This new filling is the result of attaching the handle $(H_{\Sigma_L},\omega_\beta)$ to $(W,\omega)$ along $i_1(\Sigma_L)$ and $i_2(\Sigma_L)$.

\section{Proof of Theorem \ref{main-theorem}}\label{sec:main-theorem}
Throughout this section we take $(M,\xi)$ to be a contact manifold satisfying the hypotheses of Theorem \ref{main-theorem}.  Let $(W,\omega)$ be a strong filling of $(M,\xi)$ and $\Sigma_g$ a splitting surface of genus $g$, with dividing set $\Gamma_{\Sigma_g}=c_1\cup\cdots\cup c_{g+1}$.  There are attaching arcs $\alpha^{\pm}_1,\ldots,\alpha^{\pm}_g$ and associated bypasses $D^{\pm}_1,\ldots,D^{\pm}_g$ as described in the definition of splitting surfaces.\\

We will denote by $(\widehat{W},\widehat{\omega})$ the completion of $(W,\omega)$, obtained by attaching the positive end $([0,\infty)\times M,d(e^t\alpha))$ of the symplectization of $M$.  We take $J$ to be an almost complex structure on $\widehat{W}$ adapted to the contact form $\alpha$ for $(M,\xi)$.  That is, $J$ is translation invariant, $J\xi=\xi$, and $J\partial_t=R_\alpha$, where $t$ is the $[0,\infty)$-coordinate on the symplectization and $R_\alpha$ is the Reeb vector field for $\alpha$.\\

We will prove Theorem \ref{main-theorem} by adapting the proof of \cite[Theorem 1.1]{menke2018jsj}.  Specifically, our goal is to use $\Sigma_g$ to construct a 1-parameter family $\mathcal{S}$ which sweeps out a properly embedded handlebody in $(\widehat{W},\widehat{\omega})$.  Removing this handlebody from $(W,\omega)$ will leave us with the desired manifold $(W',\omega')$.\\

Because our proof is adapted from \cite{menke2018jsj}, many of our lemmas are arbitrary-genus analogues of lemmas found there.  Some of these require new proofs, while others, such as the following standardization of the contact form on $M$, are genus-independent and therefore survive unaltered.

\begin{lemma}[{\cite[Lemma 3.1]{menke2018jsj}}]\label{lemma:dividing-orbits}
There is a choice of contact form on a neighborhood of $\Sigma_g$ such that the components $\Gamma_{\Sigma_g}$ are non-degenerate elliptic Reeb orbits of Conley-Zehnder index 1 with respect to the framing induced by $\Sigma_g$.
\end{lemma}

Denote the Reeb orbits constructed in Lemma \ref{lemma:dividing-orbits} by $e_1,\ldots,e_{g+1}$, with $e_{g+1}$ containing the endpoints of $\alpha^\pm_1,\ldots,\alpha^\pm_g$ and $e_i$ the dividing curve straddled by $\alpha^\pm_i$.  Menke's proof of Lemma \ref{lemma:dividing-orbits} produces an explicit model for $\Sigma_g$ with these orbits comprising the dividing set, and this model is depicted in Figure \ref{fig:dividing-curves}.

\begin{figure}
\centering
\Large
\def\svgwidth{0.8\linewidth}
%% Accompanies image file 'dividing-curves.pdf' (pdf, eps, ps)
\begingroup%
  \makeatletter%
  \providecommand\color[2][]{%
    \errmessage{(Inkscape) Color is used for the text in Inkscape, but the package 'color.sty' is not loaded}%
    \renewcommand\color[2][]{}%
  }%
  \providecommand\transparent[1]{%
    \errmessage{(Inkscape) Transparency is used (non-zero) for the text in Inkscape, but the package 'transparent.sty' is not loaded}%
    \renewcommand\transparent[1]{}%
  }%
  \providecommand\rotatebox[2]{#2}%
  \newcommand*\fsize{\dimexpr\f@size pt\relax}%
  \newcommand*\lineheight[1]{\fontsize{\fsize}{#1\fsize}\selectfont}%
  \ifx\svgwidth\undefined%
    \setlength{\unitlength}{498.96883266bp}%
    \ifx\svgscale\undefined%
      \relax%
    \else%
      \setlength{\unitlength}{\unitlength * \real{\svgscale}}%
    \fi%
  \else%
    \setlength{\unitlength}{\svgwidth}%
  \fi%
  \global\let\svgwidth\undefined%
  \global\let\svgscale\undefined%
  \makeatother%
  \begin{picture}(1,0.3777792)%
    \lineheight{1}%
    \setlength\tabcolsep{0pt}%
    \put(0,0){\includegraphics[width=\unitlength,page=1]{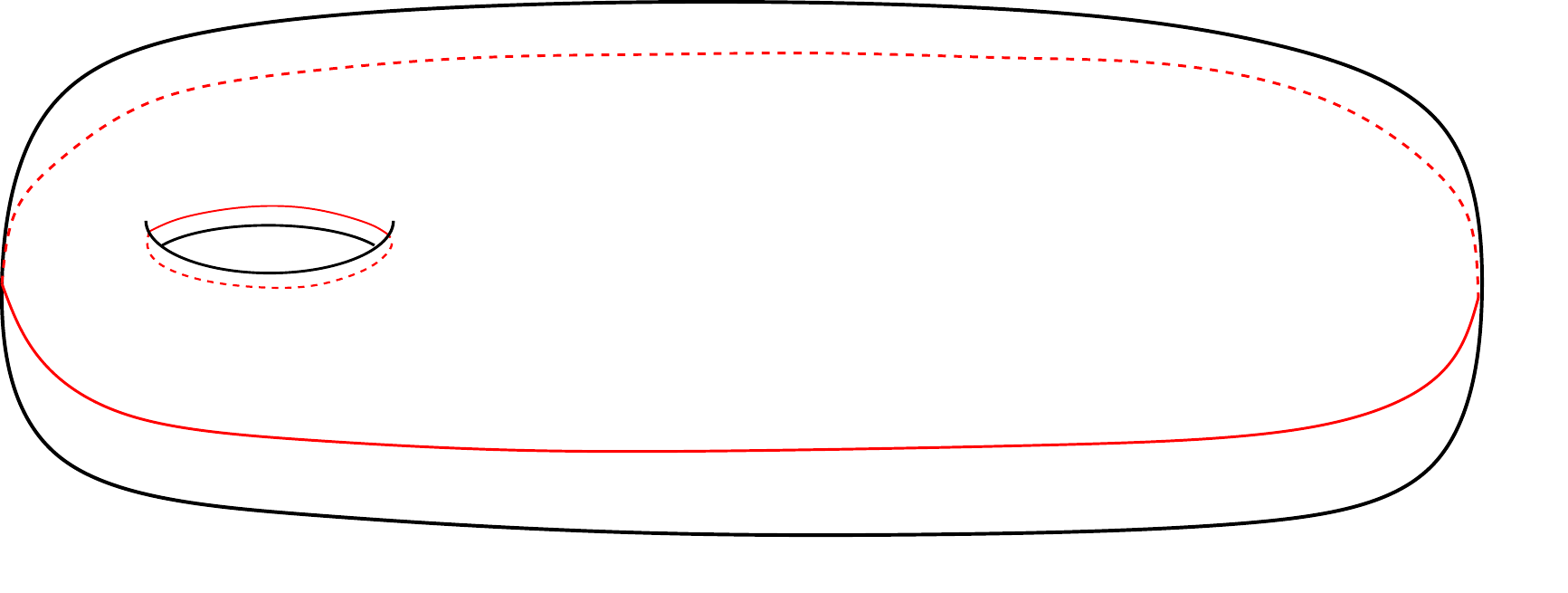}}%
    \put(0.16177014,0.25619264){\color[rgb]{1,0,0}\makebox(0,0)[lt]{\lineheight{1.25}\smash{\begin{tabular}[t]{l}$e_1$\end{tabular}}}}%
    \put(0,0){\includegraphics[width=\unitlength,page=2]{dividing-curves.pdf}}%
    \put(0.37220412,0.25619264){\color[rgb]{1,0,0}\makebox(0,0)[lt]{\lineheight{1.25}\smash{\begin{tabular}[t]{l}$e_2$\end{tabular}}}}%
    \put(0.78405349,0.25619264){\color[rgb]{1,0,0}\makebox(0,0)[lt]{\lineheight{1.25}\smash{\begin{tabular}[t]{l}$e_g$\end{tabular}}}}%
    \put(0.54355741,0.10588251){\color[rgb]{1,0,0}\makebox(0,0)[lt]{\lineheight{1.25}\smash{\begin{tabular}[t]{l}$e_{g+1}$\end{tabular}}}}%
  \end{picture}%
\endgroup%

\caption{A splitting surface $\Sigma_g$ with dividing curves $e_1,\ldots,e_{g+1}$, each of which is an elliptic orbit with Conley-Zehnder index 1.  Some of the attaching arcs are also depicted.}
\label{fig:dividing-curves}
\end{figure}

\begin{lemma}\label{lemma:description-of-orbits}
Let $\Sigma_g\subset (M,\xi)$ be a splitting surface of genus $g>1$, with dividing set $e_1\cup\cdots\cup e_{g+1}$ and bypasses $D^\pm_1,\ldots,D^\pm_g$ as described above.  There is a one-sided neighborhood
\[
N = N(\Sigma_g\cup D^+_1\cup \cdots\cup D^+_g)
\]
and an extension of the contact form $\alpha$ chosen in Lemma \ref{lemma:dividing-orbits} to $N$.  This neighborhood contains contact 1-handles $N_1^i$, contact 2-handles $N_2^i$, and surfaces with corners $\Sigma^{i-1}_g,\Sigma^i_{g+1}$ of genus $g$ and $(g+1)$, respectively, for $i=1,\ldots,g$.  Moreover,
\begin{enumerate}
	\item the boundary $\partial N$ is given by $\Sigma_g$ and $\tilde{\Sigma}$, where $\tilde{\Sigma}$ is another convex surface of genus $g$, with dividing set given by elliptic orbits $\tilde{e}_1,\ldots,\tilde{e}_{g+1}$;
	
	\item $\Sigma^0_g=\Sigma_g$, and for $i=1,\ldots,g-1$, $\Sigma^i_g$ meets $\tilde{\Sigma}$ in the orbits $\tilde{e}_1,\ldots,\tilde{e}_i$, meets $\Sigma_g$ in the orbits $e_{i+1},\ldots,e_g$, and has dividing set given by these orbits, along with an elliptic orbit $e^{i+1}_{g+1}$;

	\item for $i=1,\ldots,g$ we have a neighborhood
	\[
	N(\Sigma^{i-1}_g\cup D^+_i) = N_1^i \cup_{\Sigma_{g+1}^i} N_2^i,
	\]
	with $\Sigma^i_{g+1}$ containing the orbits $\tilde{e}_1,\ldots,\tilde{e}_{i-1},e_i,\ldots,e_g,e^{i+1}_{g+1}$, as well as the elliptic orbit $\overline{e}_i$;
	
	\item all of the elliptic orbits listed have Conley-Zehnder index 1;
	
	\item the Reeb vector field $R_\alpha$ is positively (negatively) transverse to the positive (negative) region of each of the surfaces listed;
	
	\item there are hyperbolic orbits $h^i_{g+1},\tilde{h}_i$ in $N_1^i$ and $N_2^i$, respectively, which have Conley-Zehnder index 0 with respect to $\Sigma_g$;
	
	\item if $\gamma$ is any other Reeb orbit in $N$ and $\bar{\gamma}$ is any of $e_i,h_{g+1}^i$, or $\tilde{h}_i$, then\label{part:action-bound}
	\[
	\mathcal{A}(\bar{\gamma}) < \mathcal{A}(\bar{e}_j),\mathcal{A}(\tilde{e}_j) \ll\mathcal{A}(\gamma),
	\]
	for all $j$.  In particular, $\mathcal{A}(\gamma)$ is sufficiently large as to prohibit the existence of a pseudoholomorphic curve in the symplectization of $M$ from having $\gamma$ among its negative ends while its positive ends form a subset of the curves listed.\label{vaugon-action-bounds}
\end{enumerate}
\end{lemma}
\begin{proof}
As in the proof of \cite[Lemma 3.2]{menke2018jsj}, we obtain the neighborhood $N$ by successively attaching the contact handles $N_1^i$ and $N_2^i$, and we extend $\alpha$ to $N$ by extending this contact form to each of these handles.\\

The first contact handle we attach, $N_1^1$, corresponds to the bypass $D^+_1$, which has its endpoints on $e_{g+1}$.  Attaching this handle requires a convex-to-sutured boundary modification, which introduces the hyperbolic orbit $h^1_{g+1}$.  We then apply a sutured-to-convex boundary modification before attaching $N_2^1$.  The result is an extension of $\alpha$ to the neighborhood $N(\Sigma_g\cup D^+_1)$ as described, and we repeat this process inductively to obtain $N$.  We choose our extension of $\alpha$ across each 1-handle so that the actions of $\overline{e}_i$ and $e_{g+1}^{i+1}$ are much larger than those of $e_i, e^i_{g+1},$ and $h^i_{g+1}$.  The fact that all other Reeb orbits intersecting $N$ have sufficiently large action as to be irrelevant follows from \cite[Theorem 2.1]{vaugon2015reeb}.
\end{proof}

A schematic of the neighborhood $N(\Sigma^{i-1}_g\cup D^+_i)$ is depicted in Figure \ref{fig:neighborhood-orbits}.

\begin{figure}
\centering\Large
\def\svgwidth{0.7\linewidth}
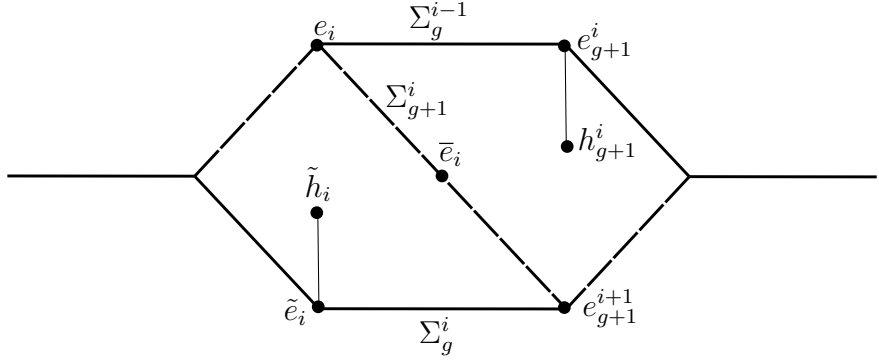
\caption{Orbits in the neighborhood $N(\Sigma^{i-1}_g\cup D^+_i)$.  The heavily shaded curves represent the walls identified in Lemma \ref{lemma:walls}.  The segments on the far left and right are included in $\Sigma_g^{i-1},\Sigma_g^i$, and $\Sigma_{g+1}^i$.  The middle (dashed) segment is included only in $\Sigma_{g+1}^i$, while the dashed segments in the upper left and lower right are also included in $\Sigma_g^{i-1}$ and $\Sigma_g^i$, respectively.}
\label{fig:neighborhood-orbits}
\end{figure}

As stated above, we will build a 1-parameter family of holomorphic curves in $\widehat{W}$ that will sweep out a handlebody of genus $g$.  The splitting surface $\Sigma_g$ will help us do this by providing targets $R_{\pm}(\Sigma_g)$ for which our family can aim at its ends.  That is, our 1-parameter family will have its ends in the symplectization part $[0,\infty)\times M$ of $\widehat{W}$, and we want the projection $\pi\colon[0,\infty)\times M\to M$ to take the ends of our family to the regions $R_{\pm}(\Sigma_g)$.  The first step towards building our 1-parameter family is then to lift $R_{\pm}(\Sigma_g)$ to embedded holomorphic curves
\[
u_{\pm}\colon S^2\setminus\{p_1,\ldots,p_{g+1}\}\to[0,\infty)\times M.
\]
We can obtain these lifts by employing the following strategy: for each $1\leq i\leq g+1$ we construct a holomorphic half-cylinder
\[
u_i\colon[0,\infty)\times S^1\to \mathbb{R}\times M
\]
which is positively asymptotic to $e_i$.  These half-cylinders project under $\pi$ to collar neighborhoods of $e_1,\ldots,e_{g+1}$ in $R_{\pm}(\Sigma_g)$, the deletion of which leaves $R'_{\pm}$, a 2-dimensional Weinstein domain.  Our lifting problem is then solved if we can lift $R'_{\pm}$ to a holomorphic curve in $\mathbb{R}\times M$ and then glue the holomorphic half-cylinders $u_1,\ldots,u_{g+1}$ to the boundary.  The following lemma, proved in \cite{menke2018jsj}, allows us to lift $R'_{\pm}$.

\begin{lemma}[{\cite[Lemma 3.4]{menke2018jsj}}]\label{lemma:weinstein-domains}
Let $(B,\beta=-df\circ J)$ be a 2-dimensional Weinstein domain, where $f\colon B\to\mathbb{R}$ is a Morse function such that $\partial B$ is a level set of $f$, and let $\alpha=dt+\beta$ be a contact form on $[-\epsilon,\epsilon]\times B$, where $t$ is the coordinate on $[-\epsilon,\epsilon]$.  Then there is an adapted almost complex structure on $\mathbb{R}\times[-\epsilon,\epsilon]\times B$ such that we can lift $B$ to a holomorphic curve by the map $u(\mathbf{x})=(f(\mathbf{x}),0,\mathbf{x})$.
\end{lemma}

The construction of the holomorphic half-cylinders $u_1,\ldots,u_{g+1}$ and the gluing of these to our lifts is also carried out in \cite{menke2018jsj}; this establishes the following result.

\begin{lemma}[{\cite[Lemma 3.5]{menke2018jsj}}]\label{lemma:lifts}
There are embedded holomorphic curves
\[
u_{\pm}\colon S^2\setminus\{p_1,\ldots,p_{g+1}\}\to[0,\infty)\times M
\]
such that
\begin{enumerate}
	\item both are Fredholm regular with index 2 and positively asymptotic to $e_1,\ldots,e_{g+1}$;
	\item under the projection $\pi\colon[0,\infty)\times M\to M$ we have $\im(\pi\circ u_{\pm})=R_\pm(\Sigma_g)$.
\end{enumerate}
\end{lemma}

The same holomorphic half-cylinder strategy is used in \cite{menke2018jsj} to prove the next result that we will need.  Because $\Sigma_g$ is a splitting surface, it admits collections of bypasses $\D_+$ and $\D_-$ from opposite sides, and Lemma \ref{lemma:description-of-orbits} describes the orbits that appear in a neighborhood $N(\Sigma_g\cup \D_+\cup \D_-)$.  Specifically, Lemma \ref{lemma:description-of-orbits} gives a list of relevant orbits in $N(\Sigma_g\cup\D_+)$, and produces a corresponding list in $N(\Sigma_g\cup\D_-)$.  We distinguish the orbits in $N(\Sigma_g\cup\D_-)$ from those in $N(\Sigma_g\cup\D_+)$ with a prime (e.g., $\overline{e}_i'$ instead of $\overline{e}_i$).  Some of these orbits are represented diagrammatically in Figure \ref{fig:double-neighborhood-orbits}.  In Lemma \ref{lemma:description-of-orbits}, the attachment of the bypass $D^+_i$ was accomplished by attaching the contact handles $N_1^i$ and $N_2^i$; we use the handles $(N_2^i)'$ and $(N_1^i)'$ to attach $D^-_i$.  The same approach used to prove Lemma \ref{lemma:lifts} produces a collection of holomorphic curves which project to $N(\Sigma_g\cup \D_+\cup \D_-)$ and will be useful to us in constructing our 1-parameter family.

\begin{lemma}[{\cite[Lemma 3.6]{menke2018jsj}}]\label{lemma:walls}
For $i=1,\ldots,g$, there are embedded holomorphic curves
\[
\bar{u}_{\pm,i},\bar{u}'_{\pm,i}\colon S^2\setminus\{p_1,\ldots,p_{g+2}\}\to [0,\infty)\times N(\Sigma_g\cup \D_+\cup \D_-)
\]
and
\[
\tilde{u}_{\pm,i},\tilde{u}'_{\pm,i}\colon S^2\setminus\{p_1,\ldots,p_{g+1}\}\to [0,\infty)\times N(\Sigma_g\cup \D_+\cup \D_-),
\]
all Fredholm regular of index 2, all positively asymptotic to $\tilde{e}_1,\ldots,\tilde{e}_{i-1},e_{i+1},\ldots,e_g$, and additionally
\begin{enumerate}
	\item $\bar{u}_{\pm,i}$ is positively asymptotic to $e_i,\bar{e}_i$, and $e_{g+1}^{i+1}$;
	\item $\bar{u}'_{\pm,i}$ is positively asymptotic to $e_i,\bar{e}'_i$, and $(e_{g+1}^{i+1})'$;
	\item $\tilde{u}_{\pm,i}$ is positively asymptotic to $\tilde{e}_i$ and $e_{g+1}^{i+1}$;
	\item $\tilde{u}'_{\pm,i}$ is positively asymptotic to $\tilde{e}'_i$ and $(e_{g+1}^{i+1})'$.
\end{enumerate}
Curves with the same asymptotic ends are distinguished by whether their projections to $\Sigma_g\subset M$ agree with that of $R_+(\Sigma_g)$ or $R_-(\Sigma_g)$.
\end{lemma}

The holomorphic curves given by Lemma \ref{lemma:walls} serve as ``walls" between the contact handles that have been attached to $\Sigma_g$ and will be used to enumerate certain holomorphic curves appearing in the symplectization $\mathbb{R}\times M$.  Some of these walls are depicted as heavily shaded curves in in Figure \ref{fig:neighborhood-orbits}.\\

\begin{figure}
\centering\Large
\def\svgwidth{0.7\linewidth}
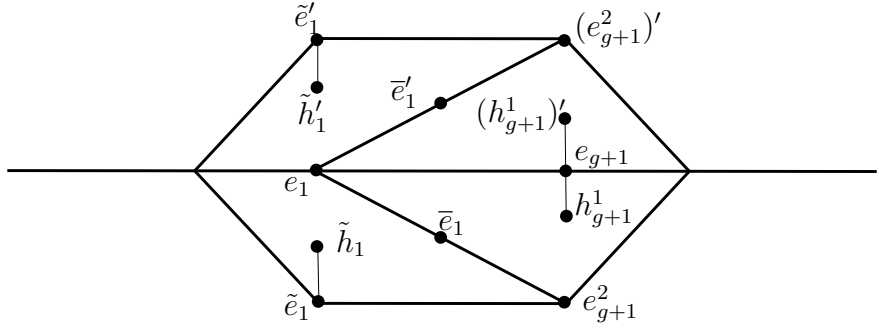
\caption{Orbits in the neighborhood $N(\Sigma_g\cup D^+_1\cup D^-_1)$.  The heavily shaded curves represent some of the walls identified in Lemma \ref{lemma:walls}.}
\label{fig:double-neighborhood-orbits}
\end{figure}

Let $\mathcal{M}(e_1,\ldots,e_{g+1})$ be the index-2 moduli space of curves $u\colon S^2\setminus\{p_1,\ldots,p_{g+1}\}\to\mathbb{R}\times M$ which are positively asymptotic to $e_1,\ldots,e_{g+1}$ and homologous to either $u_+$ or $u_-$.  This space admits an obvious translation action by $\mathbb{R}$, and the following lemma describes the compactification of $\mathcal{M}(e_1,\ldots,e_{g+1})/\mathbb{R}$.

\begin{lemma}\label{lemma:boundary-curves}
The compactification $\overline{\mathcal{M}(e_1,\ldots,e_{g+1})/\mathbb{R}}$ contains a pair of closed intervals $\mathcal{N}_{\pm}$ such that
\begin{enumerate}
	\item $\mathcal{N}_{\pm}$ contains the equivalence class of $u_{\pm}$;
	\item the boundary $\partial\mathcal{N}_{\pm}$ contains a two-level holomorphic building with top level $v_{1,\pm}$ a cylinder positively asymptotic to $e_{g+1}$ and negatively asymptotic to $h^1_{g+1}$, and with bottom level $v_{0,\pm}$ positively asymptotic to $e_1,\ldots,e_g,h^1_{g+1}$;
	\item the other boundary element of $\partial\mathcal{N}_{\pm}$ is a two-level holomorphic building with top level $v'_{1,\pm}$ a cylinder positively asymptotic to $e_{g+1}$ and negatively asymptotic to $(h^1_{g+1})'$, and with bottom level $v'_{0,\pm}$ positively asymptotic to $e_1,\ldots,e_g,(h^1_{g+1})'$.
\end{enumerate}
\end{lemma}
\begin{proof}
We assume that $\mathcal{A}(e_1)=\mathcal{A}(e_2)=\cdots=\mathcal{A}(e_{g+1})$; we will use this action information as well as a description of the homology classes of the relevant curves to determine $\partial\mathcal{N}_{\pm}$.  Consider
\[
H_1(N(\Sigma_g\cup \mathbf{D}_+\cup \mathbf{D}_-)) \simeq H_1(\Sigma_g) \simeq \mathbb{Z}^{2g},
\]
and notice that we may choose curves $b_1,\ldots,b_g\subset\Sigma_g$ so that $[e_1],\ldots,[e_g],[b_1],\ldots,[b_g]$ forms a basis for $H_1(\Sigma_g)$.  Moreover, the curve $b_i$ is chosen so that if the attaching arc $\alpha_i^\pm$ is joined with (a subarc of) the arc $a_i$ identified in the definition of a splitting surface, then the resulting closed curve is homologous to $b_i$.  After orienting the curves $b_1,\ldots,b_g$, we compute the following homology classes\footnote{The curves $b_1,\ldots,b_g$ are not canonically oriented, but we fix their orientations according to equation \ref{eq:new-homology-classes}.}
\begin{equation}\label{eq:new-homology-classes}
[\tilde{e}_i] = [e_i]-\sum_{k=1}^{g-i}[b_{i+k}],
\quad
[e^i_{g+1}] = [e^{i-1}_{g+1}] + \sum_{k=1}^{g-i}[b_{i+k}],
\quad\text{and}\quad
[\overline{e}_i]=[b_i],
\end{equation}
where $e^1_{g+1}:=e_{g+1}$.  The equation on the left is valid for $1\leq i\leq g$, the right is valid for $2\leq i\leq g+1$, and we recall that $\tilde{e}_{g+1}=e^{g+1}_{g+1}$.  Similarly,
\begin{equation}
[\tilde{e}'_i] = [e_i]+\sum_{k=1}^{g-i}[b_{i+k}],
\quad
[(e^i_{g+1})'] = [(e^{i-1}_{g+1})'] + \sum_{k=1}^{g-i}[b_{i+k}],
\quad\text{and}\quad
[\overline{e}'_i]=-[b_i],
\end{equation}
with the same conventions.  Of course $[h^1_{g+1}]=[(h^1_{g+1})']=[e_{g+1}]$, while $[h^i_{g+1}]=[e^i_{g+1}]$ and $[(h^i_{g+1})']=[(e^i_{g+1})']$ for $2\leq i\leq g$.  We also have $[\tilde{h}_i]=[\tilde{e}_i]$ for $1\leq i\leq g$.  Now suppose we have a $(k+1)$-level holomorphic building $w_k\cup w_{k-1}\cup\cdots\cup w_0$ in $\partial\mathcal{N}_{\pm}$, with top level $w_k$ and bottom level $w_0$.  Let $w^+_i$ and $w^-_i$ denote the sets of Reeb orbits to which $w_i$ is positively and negatively asymptotic, respectively.  We denote by $\mathcal{A}(w^{\pm}_i)$ the sum of the $\alpha$-actions of the Reeb orbits in $w^{\pm}_i$ and by $[w^{\pm}_i]$ the sum of their homology classes.  Of course we must have $\mathcal{A}(w^-_i)<\mathcal{A}(w^+_i)$ and $[w^-_i]=[w^+_i]$.  We also point out that the curves $\overline{u}_{\pm,i}, \overline{u}'_{\pm,i}, \tilde{u}_{\pm,i}$, and $\tilde{u}'_{\pm,i}$ are all disjoint from the curves $u_{\pm}$ and hence, by the positivity of intersections, from our holomorphic building.  In particular, these curves are disjoint from each level $w_i$.  Moreover, the projections of these curves to $M$ remain disjoint, so for each $i$, the image of $\pi\circ w_i$ is contained in a neighborhood $N_1^j$ or $(N_2^j)'$, for some $j$.\\

Now because $u_{\pm}$ is positively asymptotic to $e_1,e_2,\ldots,e_{g+1}$ we know that
\[
w^+_k \subseteq \{e_1,e_2,\ldots,e_{g+1}\}.
\]
We now consider the neighborhoods $N_1^j$ or $(N_2^j)'$ in which $\pi\circ w_k$ might land.  First, suppose that $\pi\circ w_k\subset N_1^j$ for some $j>2$.  Because $\Sigma^{j-1}_g$ meets $\Sigma_g$ in the curves $e_j,\ldots,e_g$, we have
\[
w_k^+ \subset \{e_j,\ldots,e_g\}
\quad\text{and}\quad
w_k^- \subset \{e_j,\ldots,e_g,\overline{e}_j,e^j_{g+1},h^j_{g+1},e^{j+1}_{g+1}\}.
\]
The action bounds of Lemma \ref{lemma:description-of-orbits} allow us to exclude other curves from $w_k^-$.  From equation \ref{eq:new-homology-classes} we see that the homological requirement $[w_k^+]=[w_k^-]$ can only be satisfied if we have $w_k^+=w_k^-$, and this of course violates the action requirement $\mathcal{A}(w_k^+)>\mathcal{A}(w_k^-)$.  We conclude that $\pi\circ w_k$ cannot be contained in $N_1^j$ if $j>1$.  A completely analogous argument shows that $\pi\circ w_k$ cannot be contained in $(N_2^j)'$ when $j>1$.\\

So $\pi\circ w_k$ is contained in either $N_1^1$ or $(N_2^1)'$.  In the first case we see that
\[
w_k^+ \subset \{e_1,\ldots,e_{g+1}\}
\quad\text{and}\quad
w_k^- \subset \{e_1,\ldots,e_{g+1},\overline{e}_1,h^1_{g+1},e^2_{g+1}\},
\]
again using the action bounds of Lemma \ref{lemma:description-of-orbits}.  The homological requirement $[w_k^+]=[w_k^-]$ then leads us to
\[
w_k^- = (w_k^+\setminus\{e_{g+1}\})\cup\{h^1_{g+1}\}
\quad\text{or}\quad
w_k^- = (w_k^+\setminus\{e_{g+1}\})\cup\{\overline{e}_1,e^2_{g+1}\}.
\]
The latter case is ruled out by part \ref{part:action-bound} of Lemma \ref{lemma:description-of-orbits} and the fact that $\mathcal{A}(w_k^-)<\mathcal{A}(w_k^+)$.  So if $\pi\circ w_k$ is contained in $N_1^1$, then $w_k^-=(w_k^+\setminus\{e_{g+1}\})\cup\{h^1_{g+1}\}$, and similarly if $\pi\circ w_k$ is contained in $(N_2^j)'$, then $w_k^-=(w_k^+\setminus\{e_{g+1}\})\cup\{(h^1_{g+1})'\}$.\\

An important observation at this point is that $w_k^-$ contains either $h^1_{g+1}$ or $(h^1_{g+1})'$, and thus so does $w_{k-1}^+$.  As with $\pi\circ w_k$, $\pi\circ w_{k-1}$ must be contained in a neighborhood of the form $N_1^j$ or $(N_2^j)'$.  Indeed, if $h^1_{g+1}\in w_{k-1}^+$, then $\pi\circ w_{k-1}\subset N_1^1$ and if $(h^1_{g+1})'\in w_{k-1}^+$, then $\pi\circ w_{k-1}\subset (N_2^1)'$.  We now consider these two cases.\\

If $\pi\circ w_{k-1}$ is contained in $N_1^1$ then
\[
w^-_{k-1}\subseteq \{e_1,\ldots,e_{g+1},\overline{e}_1,h^1_{g+1},e^2_{g+1}\}.
\]
Now $w^+_{k-1}$ must contain $h^1_{g+1}$, must be homologous to $w^-_{k-1}$, and must satisfy $\mathcal{A}(w^-_{k-1})<\mathcal{A}(w^+_{k-1})$.  The first two conditions are satisfied if
\[
w^+_{k-1} = \{e_1,\ldots,e_g,h^1_{g+1}\}
\quad\text{and}\quad
w^-_{k-1} = \emptyset
\]
or if
\[
w^+_{k-1} = \{h^1_{g+1}\} \cup \bar{w}
\quad\text{and}\quad
w^-_{k-1} = \{\bar{e}_1,e^2_{g+1}\} \cup \bar{w}
\]
for some $\bar{w}\subseteq \{e_1,\ldots,e_{g}\}$.  However, the latter case is prohibited by the action bound, so we conclude that $w^-_{k-1}=\emptyset$, meaning that our building has height two.  All that remains is to verify that the top level of our building is a cylinder.  To see that this is the case, notice that $w_{k-1}$ must be connected, since $w_{k-1}^+=\{e_1,\ldots,e_g,h^1_{g+1}\}$ and the only null-homologous combination of these positive ends is $e_1+\cdots+e_g+h^1_{g+1}$.  So if $w_k$ has more than one negative end, then the building $w_k\cup w_{k-1}$ has nonzero genus.  Of course this is impossible, since all of the curves in $\mathcal{M}(e_1,\ldots,e_{g+1})/\mathbb{R}$ are planar.  So $w_k$ is a cylinder with positive end $e_{g+1}$ and negative end $h^1_{g+1}$, as desired.\\

If instead the image of $\pi\circ w_{k-1}$ is contained in $(N_2^1)'$, then the same considerations lead us to conclude that $w_k$ is a cylinder with positive end $e_{g+1}$ and negative end $(h^1_{g+1})'$, and that $w_{k-1}$ is positively asymptotic to $e_1,\ldots,e_g,h^1_{g+1}$, with no negative ends.  We thus define $v_{0,\pm}=w_{k-1}$ and $v_{1,\pm}=w_k$ in the case that $\pi\circ w_{k-1}$ is contained in $N_1^1$ and define $v'_{0,\pm}=w_{k-1}$ and $v'_{1,\pm}=w_{k}$ in the case that $\pi\circ w_{k-1}$ is contained in $(N_2^1)'$.
\end{proof}

Now let $\mathcal{M}_{\widehat{W}}(e_1,\ldots,e_g,h^1_{g+1})$ be the index-1 moduli space of holomorphic curves in $\widehat{W}$ which are positively asymptotic to $e_1,\ldots,e_g,h^1_{g+1}$ and represent the same homology class as $v_{0,+}$ or $v_{0,-}$, the curves identified (up to translation) in Lemma \ref{lemma:boundary-curves}.  The following lemma will allow us to use this moduli space to interpolate between $v_{0,+}$ and $v_{0,-}$, producing what will serve as the middle part of our 1-parameter family.

\begin{lemma}\label{lemma:middle-interval}
One component of the compactification $\overline{\mathcal{M}_{\widehat{W}}(e_1,\ldots,e_g,h^1_{g+1})}$ is a closed interval $I$ with $\partial I=\{v_{0,+},v_{0,-}\}$.
\end{lemma}
\begin{proof}
We denote $\mathcal{M}_{\widehat{W}}(e_1,\ldots,e_g,h^1_{g+1})$ by $\mathcal{M}_{\widehat{W}}$ and investigate the objects that could appear in the boundary of the compactification of $\mathcal{M}_{\widehat{W}}$.  Because this is an index-1 family, the compactification will not contain any nodal curves, and the only possible boundary elements are holomorphic buildings in the symplectization end of $\widehat{W}$.  Suppose we have such a building, and let
\[
w\colon S^2\setminus\{p_1,\ldots,p_k\}\to\mathbb{R}\times M
\]
be its topmost level.  As in the proof of Lemma \ref{lemma:boundary-curves}, the curves $\bar{u}_{\pm,i},\bar{u}'_{\pm,i},\tilde{u}_{\pm,i}$, and $\tilde{u}'_{\pm,i}$ are all disjoint from elements of $\mathcal{M}_{\widehat{W}}$ and hence, by the positivity of intersections, from $w$.  So the image of the projection $\pi\circ w$ must be contained in one of the neighborhoods $N_1^j,N_2^j,(N_1^j)',(N_2^j)'$ identified above.  We claim that this is only possible if $w$ is positively asymptotic to $e_1,\ldots,e_g,h^1_{g+1}$ and has no negative ends.\\

We first show that $\pi\circ w$ cannot be contained in a neighborhood of the form $N_2^j$ or $(N_1^j)'$.  To this end, suppose that $\pi\circ w$ is contained in $N_2^j$.  Then
\[
w^+ \subset\{e_1,\ldots,e_g\}
\quad\text{and}\quad
w^- \subset\{e_1,\ldots,e_g,\overline{e}_j,\tilde{e}_j,\tilde{h}_j,e^{j+1}_{g+1}\}.
\]
But the homology classes computed in equation \ref{eq:new-homology-classes} tell us that curves chosen in this way can only satisfy $[w_k^+]=[w_k^-]$ if in fact $w_k^+=w_k^-$.  Of course this violates the inequality $\mathcal{A}(w_k^+)>\mathcal{A}(w_k^-)$, and we see that $\pi\circ w$ cannot be contained in $N_2^j$ for any $j$.  The same reasoning shows that $\pi\circ w$ also cannot be contained in a neighborhood of the form $(N_1^j)'$.\\

Just as in the proof of Lemma \ref{lemma:boundary-curves}, the projection $\pi\circ w$ of the topmost level $w$ cannot be contained in $N^j_1$ or $(N_2^j)'$ if $j>1$.  These leaves two possibilities --- either $\pi\circ w$ is contained in $N_1^1$, or in $(N_2^1)'$ --- which we now consider.\\

Suppose that the image of $\pi\circ w$ is contained in $N_1^1$, meaning that
\[
w^+\subseteq \{e_1,\ldots,e_g,h^1_{g+1}\}
\quad\text{and}\quad
w^-\subseteq \{e_1,\ldots,e_{g+1},h^1_{g+1},\bar{e}_1,e^2_{g+1}\}.
\]
Again we must have $[w^+]=[w^-]$.  Because we could have $[w^+]=0$, it is possible that $w^-$ is empty, and we have a holomorphic building of height one.  Suppose this is not the case.  Because $[h^1_{g+1}]=[e^1_{g+1}]+[\overline{e}_1]=[e_{g+1}]$, one homological possibility is that as we move from $w^+$ to $w^-$ we replace the curve $h^1_{g+1}$ with $e_{g+1}$ or with $e^2_{g+1}$ and $\overline{e}_1$.  That is, if $w^-\neq\emptyset$, then either $w^+=w^-$,
\[
w^+=\{h^1_{g+1}\}\cup\overline{w}
\quad\text{and}\quad
w^-=\{e_{g+1}\}\cup\overline{w}
\]
for some $\overline{w}\subseteq\{e_1,\ldots,e_{g}\}$, or
\[
w^+=\{h^1_{g+1}\}\cup\overline{w}
\quad\text{and}\quad
w^-=\{e^2_{g+1},\overline{e}_1\}\cup\overline{w}.
\]
However, all of these possibilities are prohibited by the action requirement $\mathcal{A}(w^-)<\mathcal{A}(w^+)$.  The first possibility obviously violates this requirement, while the second and third do so because $\mathcal{A}(h^1_{g+1})<\mathcal{A}(e_{g+1})<\mathcal{A}(e^2_{g+1})+\mathcal{A}(\overline{e}_1)$.  From all of this we conclude that $w^-=\emptyset$ and thus $w$ cannot be the topmost level of a building of height greater than one.  The same reasoning shows that if $\pi\circ w$ is contained in $(N_2^1)'$ then $w^-=\emptyset$.  So in any case, $w^-$ is empty, and $\pi\circ w$ is contained in either $N_1^1$ or $(N_2^1)'$.  But if $\pi\circ w$ is contained in $(N_2^1)'$, then
\[
w^+\subseteq\{e_1,\ldots,e_{g}\},
\]
so we cannot have $[w^+]=0$.  So in fact the image of $\pi\circ w$ lies in $N_1^1$, and $w$ has no negative ends.\\

So  $w$ is a holomorphic curve in the symplectization end of $\widehat{W}$ positively asymptotic to $e_1,\ldots,e_{g},h^1_{g+1}$.  In Lemma \ref{lemma:boundary-curves} we showed that there are precisely two such curves --- $v_{0,+}$ and $v_{0,-}$ --- so $w$ must be one of these two.  We conclude that
\[
\partial\overline{\mathcal{M}_{\widehat{W}}(e_1,\ldots,e_g,h^1_{g+1})} = \{v_{0,+},v_{0,-}\}.
\]
So $\overline{\mathcal{M}_{\widehat{W}}(e_1,\ldots,e_g,h^1_{g+1})}$ contains the desired component $I$.
\end{proof}

\begin{lemma}\label{lemma:1-parameter-family}
There is a 1-parameter family
\[
\mathcal{S} = \{u_t\colon S^2\setminus\{p_1,\ldots,p_{g+1}\}\to \widehat{W}~|~du_t\circ j=J\circ du_t\}_{t\in\mathbb{R}}
\]
of embedded holomorphic curves in $(\widehat{W},\widehat{\omega})$ such that
\begin{enumerate}
	\item for $t\gg 0$, the images of $u_t$ and $u_{-t}$ are contained in the symplectization part of $\widehat{W}$;\label{property:symplectization-part}
	\item for $t\gg 0$, the image of $\pi\circ u_{\pm t}$ is $R_{\pm}(\Sigma_g)$, where $\pi\colon[0,\infty)\times M\to M$ is the obvious projection;\label{property:projections}
	\item the images of $u_{t_1}$ and $u_{t_2}$ are disjoint whenever $t_1\neq t_2$.\label{property:disjoint-images}
\end{enumerate}
\end{lemma}
\begin{proof}
Consider the interval $I$ given by Lemma \ref{lemma:middle-interval}.  We take this interval to be the ``middle part" of $\mathcal{S}$ and for $t\gg 0$ we take $u_{\pm t}$ to be $v_{0,\pm}$, translated by $t+c$ in the symplectization end $[0,\infty)\times M$, where $c$ is some constant.  Property (\ref{property:symplectization-part}) follows immediately.  Because $v_{0,\pm}$ is positively asymptotic to $h^1_{g+1}$ and not $e_{g+1}$, we must isotope $\Sigma_g$ to ensure that $R_{\pm}(\Sigma_g)=\im(\pi\circ v_{0,\pm})$ and thus satisfy property (\ref{property:projections}).  Finally, notice that if $t_1\neq t_2$ are large then the images of $u_{t_1}$ and $u_{t_2}$ are disjoint; the positivity of intersections and the homotopy invariance of the intersection number tells us that in fact $u_{t_1}$ and $u_{t_2}$ are disjoint for any $t_1\neq t_2$.
\end{proof}

\begin{lemma}\label{lemma:embedded-handlebody}
The map $\iota\colon\mathbb{R}\times (S^2\setminus\{p_1,\ldots,p_{g+1}\})\to \widehat{W}$ defined by
\[
\iota(t,x) := u_t(x),
\]
with $u_t$ as identified in Lemma \ref{lemma:1-parameter-family}, is an embedding of a genus-$g$ handlebody into $\widehat{W}$.
\end{lemma}
\begin{proof}
For an arbitrary $t\in\mathbb{R}$ the curve $u_{t}$ is an embedding and thus each curve $u_{t'}$, for $t'$ near $t$, can be thought of as a section of the normal bundle $N_{u_{t}}$.  We can compute the first Chern number of this bundle according to
\[
c_1(N_{u_{t}}) = c_1(u_{t}^*T\widehat{W}) - \chi(S^2\setminus\{p_1,\ldots,p_{g+1}\}) = c_1(u_t^*T\widehat{W}) + g-1,
\]
but first we must compute $c_1(u_{t}^*T\widehat{W})$.  For this we appeal to \cite[Equation 1.1]{wendl2008automatic}, which says that
\[
2c_1(u_{t}^*T\widehat{W}) = \ind(u_{t}) + \chi(S^2\setminus\{p_1,\ldots,p_{g+1}\}) - \mu_{CZ}(u_{t}),
\]
where the last term is a signed count of the Conley-Zehnder indices of the orbits to which $u_{t}$ is asymptotic.  Then
\[
2c_1(u_{t}^*T\widehat{W}) = 1 + (1-g) - g = 2-2g,
\]
so $c_1(u_{t}) = 1-g$ and it follows that $c_1(N_{u_t})=0$.  So sections of $N_{u_{t}}$ are zero-free, meaning that $\iota$ is an embedding.
\end{proof}

The stage is now set for the construction of $(W',\omega')$, the symplectic manifold promised by Theorem \ref{main-theorem}.  This construction proceeds exactly as in \cite{menke2018jsj}, with small changes to the statements of the lemmas found there.  The strategy is to remove from $W$ the handlebody $H\subset\widehat{W}$ embedded by $\iota$ in Lemma \ref{lemma:embedded-handlebody}.  This is done in stages.  First $W$ is enlarged to $W_R:=W\cup([0,R]\times M)$, with $R$ chosen large enough that the projection of $u_{\pm t}$ to $[0,R]\times M$ is $R_{\pm}(\Sigma_g)$ minus a small collar neighborhood whenever $t\gg 0$.  From $W_R$ we remove $\tilde{N}(\Gamma_{\Sigma_g})$, a small tubular neighborhood of $\{R\}\times\Gamma_{\Sigma_g}$, leaving us with $W'_R:=W_R-\tilde{N}(\Gamma_{\Sigma_g})$.  This allows us to decompose $\partial W'_R$ into its horizontal part
\begin{equation}\label{eq:horizontal-part}
\partial_hW'_R = \partial W'_R - \partial W_R \simeq \bigsqcup_{i=1}^{g+1} (S^1\times D^2)
\end{equation}
and its vertical part $\partial_vW'_R=\partial W'_R-\partial_hW'_R$, not unlike the boundary of a Lefschetz fibration over a Weinstein domain.  Note that the deletion of $\tilde{N}(\Gamma_{\Sigma_g})$ from $W_R$ removes small collar neighborhoods from $\{R\}\times R_{\pm}$, leaving us with $\{R\}\times R'_{\pm}$.  We now begin modifying $H$ in preparation for its removal from $W'_R$.

\begin{lemma}[{\cite[Lemma 3.10]{menke2018jsj}}]\label{lemma:liouville-embedding}
There exists an embedding $\Sigma_L\times[-T,T]\hookrightarrow W'_R$ so that
\begin{enumerate}
	\item $\Sigma_L$ is a compact surface with genus 0 and $g+1$ boundary components;
	\item $\Sigma_L\times\{\pm T\}=\{R\}\times R'_{\pm}$;
	\item using the identification given in equation (\ref{eq:horizontal-part}) we have\label{item:meridians}
	\[
	\partial\Sigma_L\times\{t\} = \bigsqcup_{i=1}^{g+1} (S^1\times\gamma(t)) \subset \partial_hW'_R
	\]
	for $t\in[-T,T]$, where $\gamma(t)$ is the straight arc from $(-1,0)$ to $(1,0)$ in $D^2$.
\end{enumerate}
\end{lemma}

We denote the embedded copy of $\Sigma_L\times[-T,T]$ by $H'\subset W'_R$ and endow it with the obvious coordinates $(x,t)$.  The following two results are proven in \cite{menke2018jsj} and allow us to cut $W'_R$ along $H'$ to obtain a symplectic manifold $(W',\omega')$ that strongly fills its boundary.

\begin{lemma}[{\cite[Lemma 3.11]{menke2018jsj}}]
Let $B=[-T,T]\times[-\epsilon,\epsilon]$ with coordinates $(t,w)$.  After slight adjustments of $H'$ and $W'_R$, there exists a neighborhood $N(H')\simeq H'\times[-\epsilon,\epsilon]\subset W'_R$ and a 1-form $\lambda = \lambda_B+\lambda_{\Sigma_L}$ on $N(H')$ such that
\begin{enumerate}
	\item $\Sigma_L\times\{\pm T\}\times[-\epsilon,\epsilon]\subset\partial_vW'_R$ and $(\partial\Sigma_L)\times B\subset\partial_hW'_R$;
	\item $\lambda_{\Sigma_L}$ is the Liouville form for $R'_{\pm}$;
	\item $\lambda_B=t~dw$;
	\item $d\lambda$ is the symplectic form on $W'_R$;
	\item $\lambda$ agrees with the Liouville form on $W'_R$ near $\partial W'_R$.
\end{enumerate}
\end{lemma}

\begin{lemma}[{\cite[Lemma 3.12]{menke2018jsj}}]\label{lemma:strong-filling}
There exists a modification
\[
\lambda'=\lambda+d(tw) = 2t~dw + w~dt + \lambda_{\Sigma_L},
\]
whose Liouville vector field $Z'=2t\partial_t-w\partial_w+Z_{\Sigma_L}$ points into $N(H')$ along $w=\pm\epsilon$.
\end{lemma}

At last we define $W':=W'_R-N(H')$ and $\omega':=d\lambda'$ and from Lemma \ref{lemma:strong-filling} we conclude that $(W',\omega')$ strongly fills its boundary.  In case our original symplectic filling was exact we ask the same of $(W',\omega')$.  Once again we may appeal to \cite{menke2018jsj}, where the proof of the following lemma is genus-independent.

%the informal intuition here is that the boundary of W' is the union of the boundaries of W'_R and N(H').  the former boundary looks like M minus a thickened genus g surface, and the latter looks like a thickend genus g surface plus two disjoint handlebodies.  taking the union thus leaves us with M plus the two handlebodies, and they've been glued in where the splitting surface once stood

\begin{lemma}[{\cite[Lemma 3.13]{menke2018jsj}}]
If $(W,\omega=d\beta)$ is an exact filling, then there exists a 1-parameter family of Liouville forms $\beta_\tau$, $\tau\in[0,1]$, on $W'_R$ such that $\beta_0=\beta$ and $\beta_1=\lambda'$ on $N(H')\cap\{-\epsilon/2\leq w\leq \epsilon/2\}$.
\end{lemma}

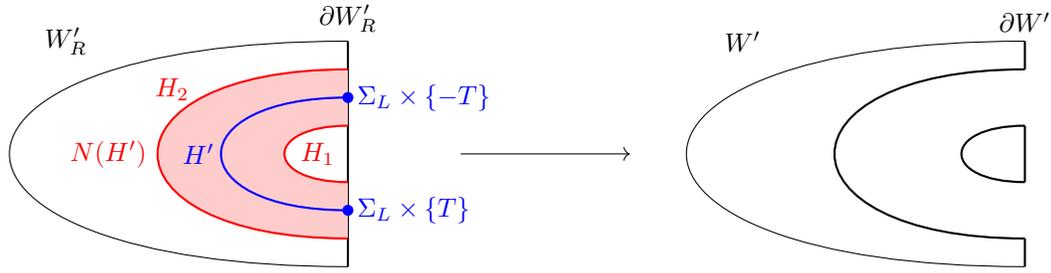
\begin{figure}
\centering
\begin{tikzpicture}[scale=0.75]
\draw (0,0) .. controls (-8,0) and (-8,4) .. (0,4);
\draw [thick] (0,0) -- (0,4);

\fill [red!20!white] (0,0.5) to (0,1.5) .. controls (-1.5,1.5) and (-1.5,2.5) .. (0,2.5) to (0,3.5) .. controls (-4.5,3.5) and (-4.5,0.5) .. (0,0.5);

\draw [blue,thick] (0,1) .. controls (-3,1) and (-3,3) .. (0,3);
\draw [red,thick] (0,0.5) .. controls (-4.5,0.5) and (-4.5,3.5) .. (0,3.5);
\draw [red,thick] (0,1.5) .. controls (-1.5,1.5) and (-1.5,2.5) .. (0,2.5);

\filldraw [blue] (0,1) circle (2.5pt);
\filldraw [blue] (0,3) circle (2.5pt);

\node at (-5,4) {$W'_R$};
\node [above] at (0,4) {$\partial W'_R$};
\node [left,red] at (-3.4,2) {$N(H')$};
\node [left,blue] at (-2.2,2) {$H'$};
\node [right,red] at (-1.0,2) {$H_1$};
\node [left,red] at (-2.65,3.15) {$H_2$};
\node [right,blue] at (0,1) {$\Sigma_L\times\{T\}$};
\node [right,blue] at (0,3) {$\Sigma_L\times\{-T\}$};

\draw [->] (2,2) -- (5,2);

\draw (12,0) .. controls (4,0) and (4,4) .. (12,4);
\draw [thick] (12,0) -- (12,0.5);
\draw [thick] (12,1.5) -- (12,2.5);
\draw [thick] (12,3.5) -- (12,4);

\draw [thick] (12,0.5) .. controls (7.5,0.5) and (7.5,3.5) .. (12,3.5);
\draw [thick] (12,1.5) .. controls (10.5,1.5) and (10.5,2.5) .. (12,2.5);

\node at (7,4) {$W'$};
\node [above] at (12,4) {$\partial W'$};
\node at (14,2) {}; %to make things centered
\end{tikzpicture}
\caption{The removal of $N(H')$ from $W'_R$.  On the right, $\partial W'$ has two connected components.}
\label{fig:boundary-intuition}
\end{figure}

Let us give an informal summary of the relationship between $\partial W'$ and $M$.  The first step in constructing $W'$ was to consider $W_R$, whose boundary is contactomorphic to $M$.  From $W_R$ we deleted a neighborhood of the dividing set of $\Sigma_g$.  This provided a decomposition of $\partial W'_R$ into its horizontal and vertical parts, but the overall effect on $\partial W_R$ was trivial.  The last step in our construction --- deleting $N(H')$ from $W_R'$ --- made the most substantive changes to the boundary.  We first identified $H'$, a handlebody in $W_R'$ which picked out for us two copies of $\Sigma_L$ in $\partial W'_R$.  Namely, $H'$ distinguished the Liouville hypersurfaces $\Sigma_L\times\{\pm T\}=\{R\}\times R'_{\pm}$.  Then $N(H')$ is a neighborhood of $H'$, part of whose boundary lies in $\partial W'_R$.  The part of $\partial N(H')$ lying in the interior of $W'_R$ consists of two disjoint copies of $H'$, and the part lying in $\partial W'_R$ includes $\Sigma_L\times\{\pm T\}$.  So deleting $N(H')$ from $W'_R$ cuts $\partial W'_R$ open along the Liouville hypersurfaces $\Sigma_L\times\{\pm T\}$ and glues in two handlebodies modeled on $H'$.  This process is depicted in Figure \ref{fig:boundary-intuition}.\\

All that remains is to use symplectic handle attachment to recover $W$ from $W'$.  To this end we observe that the neighborhood $(N(H'),d\lambda')$ we have removed from $W'_R$ is precisely the abstract symplectic handle $(H_{\Sigma_L},\omega_{\lambda_{\Sigma_L}})$ constructed from the Liouville domain $(\Sigma_L,\lambda_{\Sigma_L})$.  That is, we have obtained $W'$ from $W$ by removing a symplectic handle, and thus may recover $W$ by reattaching said handle as described in Section \ref{sec:background}.

\bibliographystyle{alpha}
\bibliography{../../references}
\end{document}